\documentclass{amsart}

\usepackage{amsmath,amssymb,amsthm, a4wide}

\usepackage{graphicx,tikz, caption}

\newtheorem{theorem}{Theorem}

\newtheorem{corollary}[theorem]{Corollary}

\newtheorem{lemma}[theorem]{Lemma}

\theoremstyle{definition}

\theoremstyle{remark}

\begin{document}

\title[]{On Zeroes of Random Polynomials\\ and Applications to Unwinding}
\keywords{Random Polynomials, Zeroes, Potential Theory, Blaschke unwinding}
\subjclass[2010]{}

\author[]{Stefan Steinerberger}
\address{Department of Mathematics, Yale University, New Haven, CT 06511, USA}
\email{stefan.steinerberger@yale.edu}

\author[]{Hau-tieng Wu}
\address{Department of Mathematics and Department of Statistical Science,
Duke University, Box 90320, Durham NC 27708, USA}
\email{hauwu@math.duke.edu}

\thanks{S.S. is supported by the NSF (DMS-1763179) and the Alfred P. Sloan Foundation. The bulk of this work was carried out while H.-T. W. was visiting the Yale REU program SUMRY and he is grateful for its hospitality.}

\begin{abstract} Let $\mu$ be a probability measure in $\mathbb{C}$ with a continuous and compactly supported density function, let $z_1, \dots, z_n$ be
independent random variables, $z_i \sim \mu$, and consider the random polynomial
$$ p_n(z) = \prod_{k=1}^{n}{(z - z_k)}.$$
{We determine the asymptotic distribution of $\left\{z \in \mathbb{C}: p_n(z) = p_n(0)\right\}$. In particular, if $\mu$ is radial around the origin, then those solutions are also distributed according to $\mu$ as $n \rightarrow \infty$. Generally, the distribution of the solutions will reproduce parts of $\mu$ and condense another part on curves.} We use these insights to study the behavior of the Blaschke unwinding series on random data.
\end{abstract}

\maketitle

\vspace{-10pt}

\section{Introduction and main results}
The purpose of this paper is to discuss an interesting phenomenon of solutions of certain random polynomial equations. In what follows, we will
assume that $\mu$ is an absolutely continuous (with respect to the Lebesgue measure) and compactly supported probability measure on $\mathbb{C}$ and that $p_n$ denotes the random polynomial
$$ p_n(z) = \prod_{k=1}^{n}{(z - z_k)},$$
where the $z_k$ are drawn independently from $\mu$ and $n \in \mathbb{N}$. Our first result is a reproducing property for radial measures $\mu$ when $n\to\infty$ (see \S 2 for the motivation that led us to this result).

\begin{figure}[h!]
\begin{minipage}[l]{.49\textwidth}
\centering
\includegraphics[width = 4.8cm]{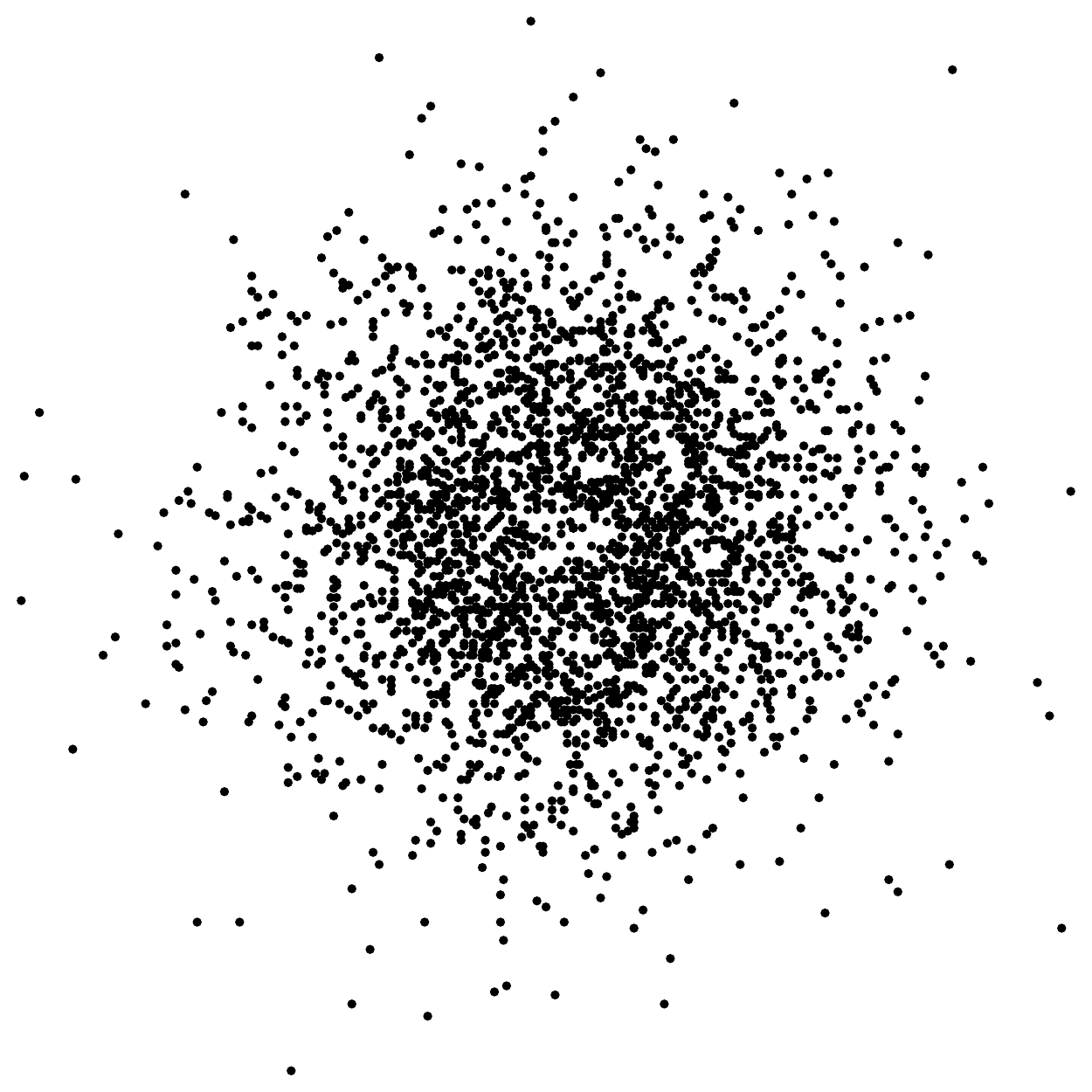} 
\end{minipage} 
\begin{minipage}[r]{.49\textwidth}
\includegraphics[width = 4.8cm]{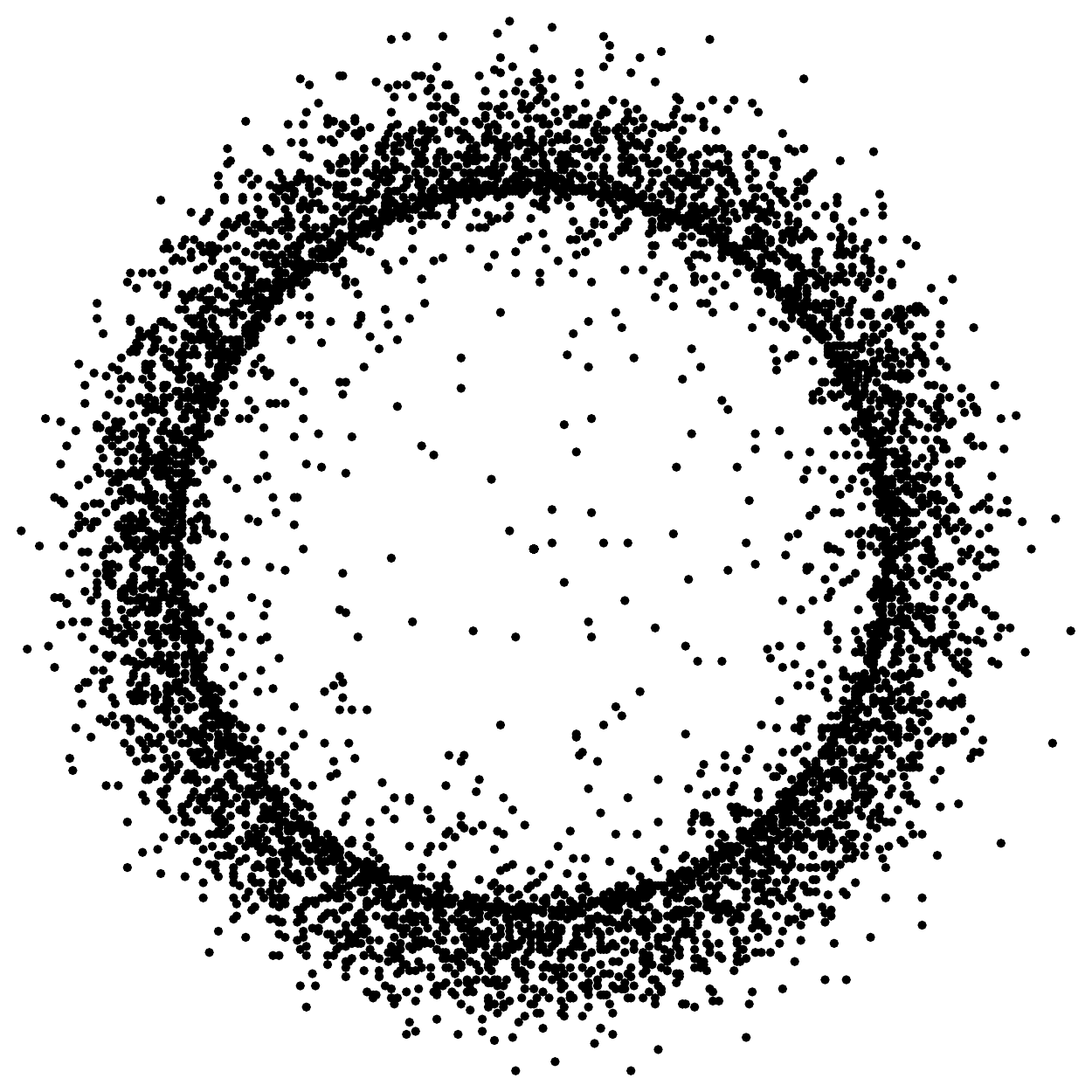} 
\end{minipage} 
\caption{Left: roots of 100 polynomials $p_{30}(z) - p_{30}(0)$ with Gaussian distributed roots are again Gaussian. Right: roots of 100 polynomials $p_{20}(z) - p_{20}(0)$ with roots uniformly distributed on the boundary of the unit disk.}
\end{figure}

\begin{theorem} Let $\mu$ be a compactly supported probability measure on $\mathbb{C}$ with a continuous, radial density function.
 Then the complex numbers $w_1, \dots, w_n$ solving $p_n(w_k) = p_n(0)$ satisfy
$$ \frac{1}{n} \sum_{k=1}^{n}{ \delta_{w_k}} \rightarrow \mu \qquad \mbox{in the sense of distributions as}~n\rightarrow \infty.$$
\end{theorem}

Theorem 1 fails for general measures but it is not
difficult to construct non-radial measures  $\mu$ that have the same property (see Theorem 2). The assumption on $\mu$ being compactly supported is clearly
not sharp, our proof immediately transfers to probability measures having a certain rate of decay at infinity. 
The result is similar in spirit to a recent result of Kabluchko \cite{kab} (proving a conjecture of Pemantle \& Rivin \cite{pe}) who showed that the distribution 
of critical points $\left\{z \in \mathbb{C}: p_n'(z) = 0\right\}$ reproduces $\mu$ for general probability measures $\mu$.
If $\mu$ is not radial, the situation is not quite as simple. We introduce two sets $A,B \subset \mathbb{C}$ (and we will keep using $A,B$ to refer to those sets
throughout the rest of the paper)
$$ A = \left\{ z \in \mathbb{C}: \int_{\mathbb{C}}{ \log{|x-z|} d\mu(x)}  >  \int_{\mathbb{C}}{ \log{|x|} d\mu(x)} \right\}$$
and
$$ B = \left\{ z \in \mathbb{C}: \int_{\mathbb{C}}{ \log{|x-z|} d\mu(x)}  =  \int_{\mathbb{C}}{ \log{|x|} d\mu(x)} \right\}.$$
A simple description of the result for the general case can be stated as follows.
\begin{theorem}[Main Result] \label{Theorem main2}
Let $\mu$ be a probability measure on $\mathbb{C}$ with a continuous and compactly supported density function.
 Then the distribution of $\left\{z \in \mathbb{C}: p_n(z) = p_n(0)\right\}$ converges to $\nu$ in distribution, where $\nu = \mu$ on $A$ and $\nu$ has measure
$1 - \mu(A)$ supported on $B$, as $n \rightarrow \infty$.
\end{theorem}
We illustrate the Theorem with a specific example (the measure is a bit more singular than what is covered by the result but it is not difficult to
see that the proof carries over to this particular case). We choose $\mu$ to be the union of the arclength measure of the boundary of two disks of radius 1 in the complex plane (one located in the
origin and one centered around 2)
$$ \mu = \frac{1}{4\pi} \left( \mathcal{H}^1\big|_{|z| = 1} \cup  \mathcal{H}^1\big|_{|z-2| = 1} \right),$$
where $\mathcal{H}^1$ is the one-dimensional Hausdorff measure.
Theorem 2 implies that the random solutions of $p_n(z) =p_n(0)$ will partially follow the original measure $\mu$ and partially concentrate along four new curves.
Details behind this example are given after the proof.

\begin{figure}[h!]
\begin{minipage}[l]{.49\textwidth}
\centering
\begin{tikzpicture}[scale=1.3]
\draw[dashed] (0,0) circle (1cm);
\draw[dashed] (2,0) circle (1cm);
   \draw [ultra thick,domain=75.5225:360-75.5225] plot ({cos(\x)}, {sin(\x)});
   \draw [ultra thick,domain=360-75.5225:360] plot ({2+cos(\x)}, {sin(\x)});
   \draw [ultra thick,domain=0:104.478] plot ({2+cos(\x)}, {sin(\x)});
   \draw [ultra thick,domain=360-104.478:360] plot ({2+cos(\x)}, {sin(\x)});
   \draw [ultra thick,domain=331:360] plot ({2*cos(\x)}, {2*sin(\x)});
   \draw [ultra thick,domain=0:29] plot ({2*cos(\x)}, {2*sin(\x)});
   \draw [ultra thick,domain=151:209] plot ({2+2*cos(\x)}, {2*sin(\x)});
\draw [ultra thick] (0.25, 0.96) to[out=10, in =180] (1,1) to[out=0, in=170] (1.75, 0.96);
\draw [ultra thick] (0.25, -0.96) to[out=350, in =180] (1,-1) to[out=0, in=190] (1.75, -0.96);
\end{tikzpicture}
\end{minipage} 
\begin{minipage}[r]{.49\textwidth}
\includegraphics[width = 7cm]{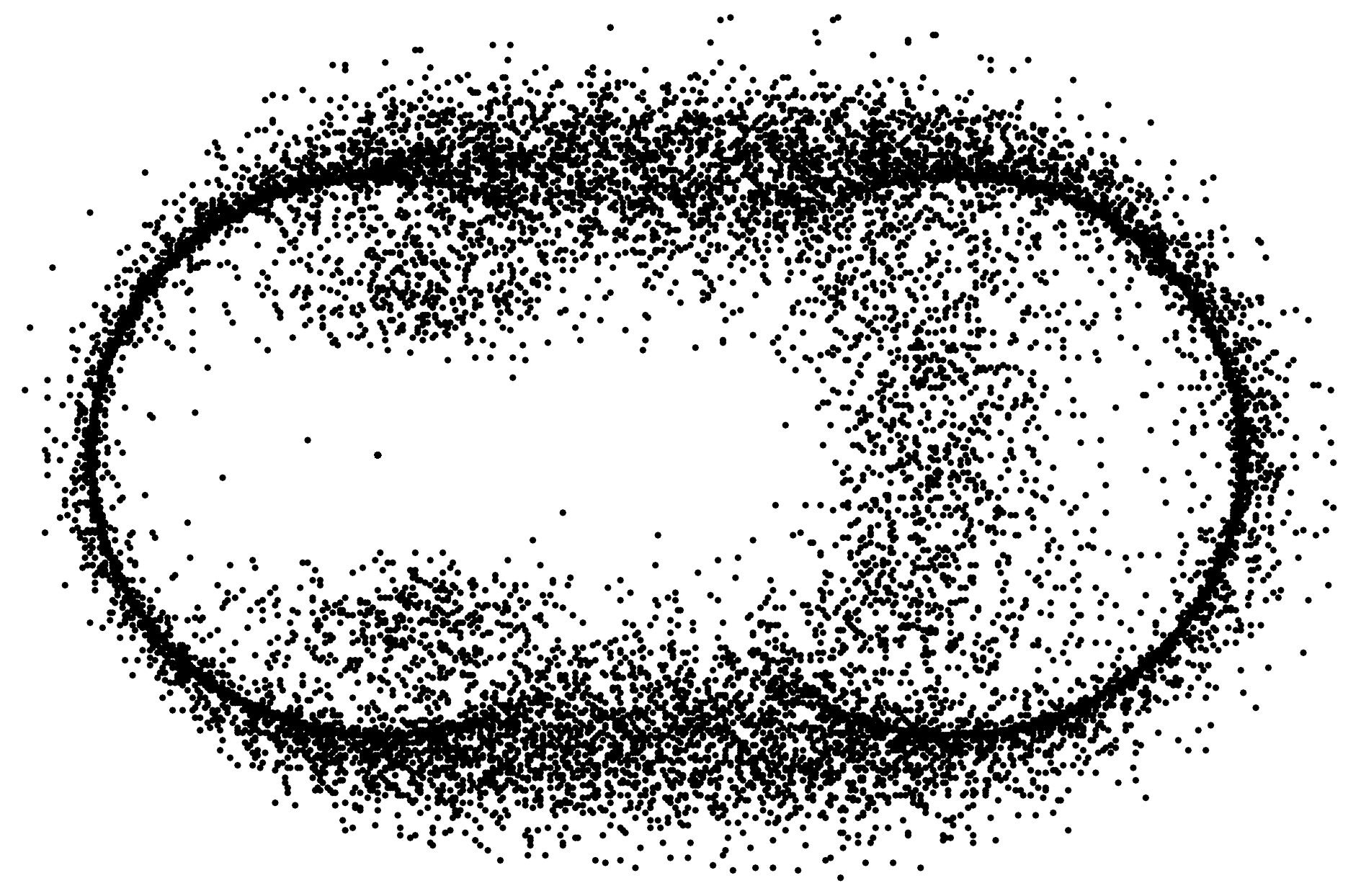} 
\end{minipage} 
\caption{Left: the support of limiting measure of solutions of $p_n(z) = p_n(0)$ as $n \rightarrow \infty$ (bold, the two circles from which roots are drawn are dashed), right: numerical example for $n=30$ with 15 roots chosen randomly from each circle. A repulsion phenomenon leads to a slight visual discrepancy in the left arc (see Theorem 3 for a more precise description).}
\end{figure}
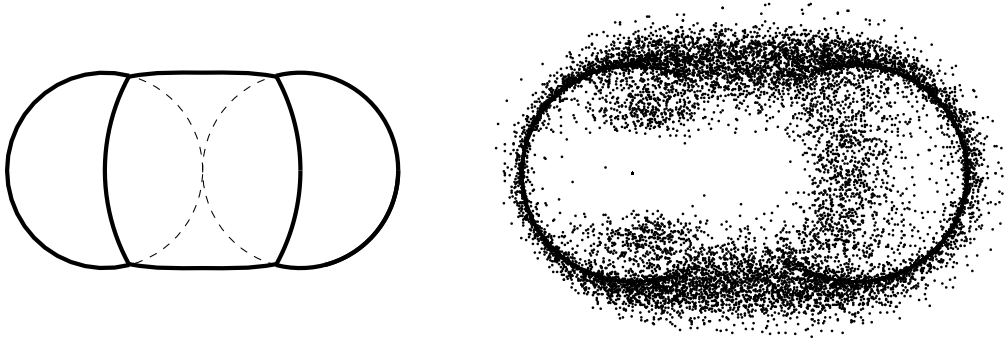

We observe that Theorem 2 does not make any claim about how the solutions of $p_n(z) = p_n(0)$ are distributed on $B$, it only states that their total mass is going to be $1-\mu(A)$. Figure 2
seems to indicate that there might indeed be parts of $B$ that will not support any part of the new measure $\nu$, however, this is misleading: there is always,
by construction, a root in the origin and a root repulsion phenomenon. This has the effect of creating a bubble around 0 in which
no roots are found; that bubble shrinks in size as the degree $n$ increases.\\

In the generic case, we can give a more precise description of the measure $\nu$ on $B$. 
Our assumptions will be that
 $\mu$ is compactly supported and $C \subset B$ is a connected subset of $B$ that is bounded away from the support of $\mu$ and 
 satisfies that 
$$ \left\| \nabla \int_{\mathbb{C}}{ \log{|x-z|} d\mu(x)} \right\| \qquad \mbox{ is uniformly bounded away from 0 on $C$. }$$
We note that compactness of the support of $\mu$ implies compactness of $B$ and thus $C$ is necessarily bounded.
We recall that $B$ is defined as a level set of the logarithmic integral, the first condition thus implies that this level set is non-degenerate and thus $C$ is necessarily a curve by the implicit function theorem.
%This assumption combined with the definition of $B$ and  immediately implies that $C$ is a curve. 
Let $\gamma$ be an arclength parametrization of the curve $C$. 
Since the support of $\mu$ is compact and $C$ is bounded away from the support of $\mu$, $\int_{\mathbb{C}} \arg(\gamma(t) - z)d\mu(z)$ is well defined as a continuous single-valued function of $t$.
We further impose the assumption that 
$$\left| \frac{\partial}{\partial t} \int_{\mathbb{C}}{\arg(\gamma(t) - z) d\mu(z)} \right| \qquad \mbox{is uniformly bounded away from 0 on $C$. }$$
Under these assumptions, we can determine the limit structure of the measure $\nu$ on the set $B$ (which we think of as a collection of curves, a level set of the logarithmic integral).

\begin{theorem}[Structure of $\nu$ on $B$] Under these assumptions, let $\gamma(t)$ be an arclength parametrization of $C$. The limiting measure of $\left\{ z \in \mathbb{C}: p_n(z) = p_n(0)\right\}$ is absolutely continuous on $C$ and given by
\begin{equation}\label{Theorem:eq1} 
\nu \big|_{C} = \frac{1}{2\pi} \left| \frac{\partial}{\partial t} \int_{\mathbb{C}}{\arg(\gamma(t) - z) d\mu(z)}\right|^{-1} d\mathcal{H}^1.
\end{equation}
Moreover, for degree $n$ sufficiently large, then with high probability the spacing between roots on $B$ becomes uniform in the sense that two consecutive roots on $C$ have distance $\sim n^{-1}$ (up to constants depending on the density) from each other with the implicit constant determined by the limiting density.
\end{theorem}
It is certainly possible to slightly extend the result to cover other cases as well. For example, here we are not necessarily assuming that $\mu$ is absolutely continuous as long as it is compactly supported and the assumptions hold.
However, it certainly already describes the generic situation fairly accurately: in particular, it allows
us to deduce that the behavior on $B$ is actually quite regular: the roots decompose into evenly spaced points (with spacing roughly $\sim n^{-1}$ and an implicit constant depending on everything).

\begin{figure}[h!]
\begin{minipage}[l]{.49\textwidth}
\includegraphics[width = 7cm]{pic3.pdf} 
\end{minipage} 
\begin{minipage}[r]{.49\textwidth}
\begin{tikzpicture}
\node at (0,0) {\includegraphics[width = 6cm]{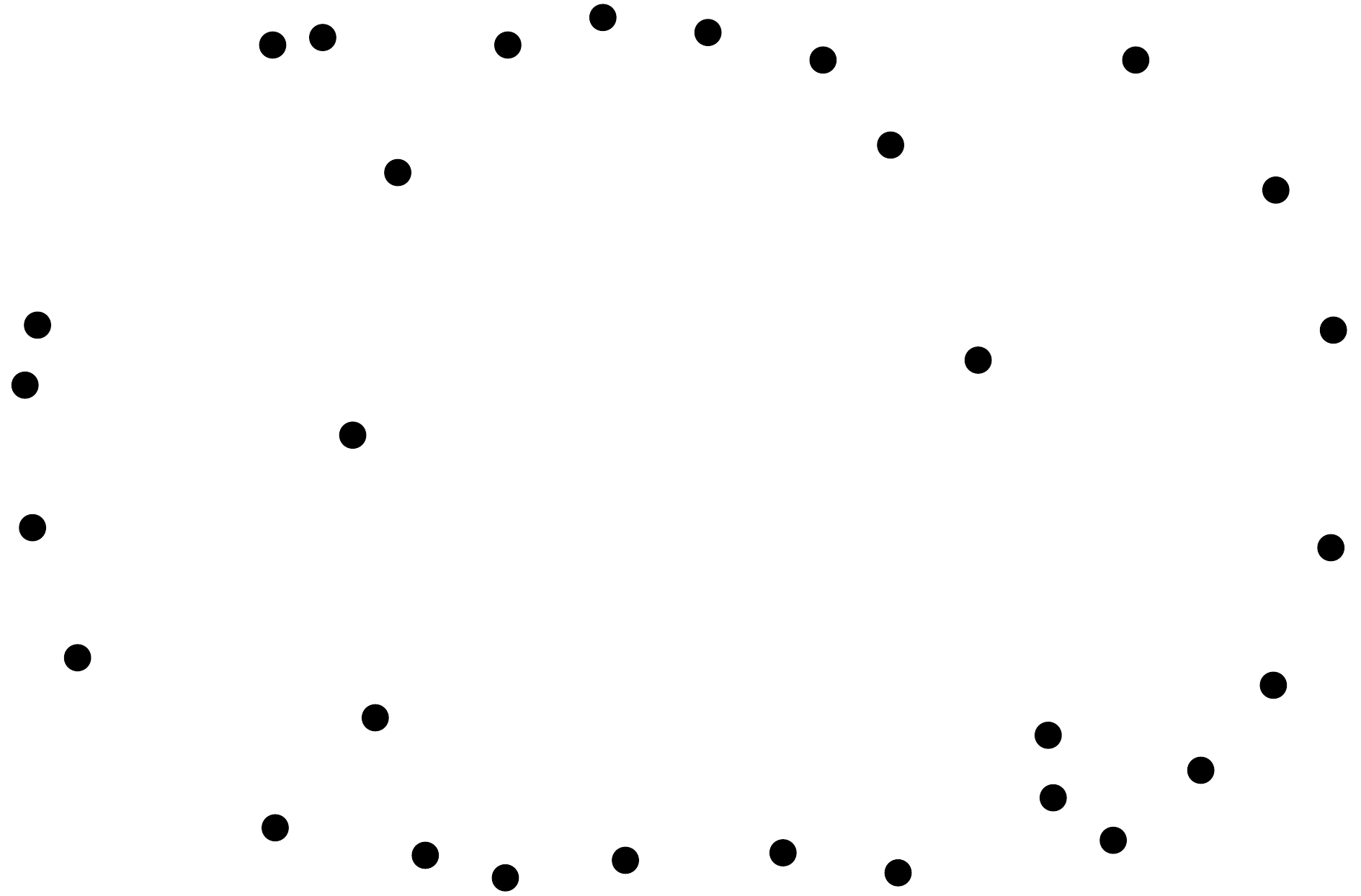}
};
\node at (-1,0) {$0$};
\draw (-1.44, 0.07) circle (0.2cm);
\end{tikzpicture}
\end{minipage} 
\caption{(Left:) the typical distribution of $p_n(z) = p_n(0)$ for $n=30$. (Right:) a single instance. Even for this rather small degree, the roots are already well-separated. The fixed root in $0$ clears out a uniform area of repulsion.}
\end{figure}

We emphasize that, in Figure 3, nothing is special about the root $z=0$. However, combining many different numerical examples has the effect of visually removing the repulsion phenomenon for all roots except the fixed one $z=0$ which is common to all numerical samples. 
We consider a simple toy example where $\mu = \delta_1$ is the deterministic measure in 1. Clearly, $p_{2n}(z) = (z-1)^{2n}$ and $p_{2n}(0) = 1$ implying that solutions of $p_{2n}(z) = p_{2n}(0)$ are given by $(z-1)^{2n} = 1$ which are equally spaced points on $|z-1|=1$. In the framework of Theorem 3, we see that
$$ \left\| \nabla \int_{\mathbb{C}}{ \log{|x-z|} d\mu(z)} \right\| = 1 \qquad \mbox{on}~|x-1|=1$$
as well as
$$ \left| \frac{\partial}{\partial t} \int_{\mathbb{C}}{\arg(\gamma(t) - z) d\mu(z)} \right| = 1$$
and the limiting measure is clearly $(2\pi)^{-1} \mathcal{H}^1$ coinciding with what is predicted by Theorem 3. 
Generically, one expects the set $B$ to be a union of bounded lines (though it can be a disk, see the proof of Theorem 1). It might be interesting to understand what happens if the measure $\mu$ decays at infinity at a certain rate. However, even for compactly supported measures, there are many fascinating open questions: whenever two of these lines meet at an angle, then clearly the gradient of the logarithmic integral vanishes and Theorem 3 does not apply: is it possible to describe the behavior of solutions of $p_n(z) = p_n(0)$ in these singular points? Moreover, one would assume that under some assumptions on the shape of $\mu$ that $B$ cannot be comprised of lines of arbitrary length. How long are these lines? How complicated can their topology be? This is related to classical questions in potential theory dating back to Maxwell (see, for example, Gabrielov, Novikov \& Shapiro \cite{gab}).

\section{Application to the Unwinding Series}
\textbf{Unwinding.}  The above results were originally motivated by a study of a nonlinear analogue of Fourier series: given a holomorphic function $f:\mathbb{C} \rightarrow \mathbb{C}$, its Blaschke factorization is given by
$$ f(z) = \left(\prod_{|\alpha| \leq 1, f(\alpha) = 0}{ \frac{z- \alpha}{1 - \overline{\alpha} z}} \right) g(z),$$
where the Blaschke product ranges over all roots inside the unit disk and $g:\mathbb{C} \rightarrow \mathbb{C}$ is holomorphic and has no roots inside the unit disk. Writing $g(z) = g(0) + (g(z) - g(0))$
produces a new holomorphic function, $g(z) - g(0)$, which has at least one root inside the unit disk. Iterating the process yields a formal expansion
$$ f(z) = a_0 B_0 + a_1 B_0 B_1 + a_2 B_0 B_1 B_2 + \dots$$
This process was introduced by Ronald R. Coifman around 1995, described in a PhD thesis of his student Michel Nahon \cite{nahon} and followed by several other researchers \cite{Saito_Letelier:2009,Healy:2009a,Healy:2009b}. It was independently discovered by T. Qian \cite{qtao} who
also studied, jointly with collaborators, different versions of the algorithm \cite{qtao2, qtao3, qtao4, qtao5}. 
There is a different line of investigation concerned with Blaschke products as a general family of orthogonal functions \cite{irr1, irr2, irr3, irr4,irr5} that we do not discuss here.
Convergence of the algorithm in the Hardy spaces $H^2$
is due to Qian (\cite{qtao}, the proof is also described in \cite{coif1}), the convergence in a large family of function spaces (including all Sobolev spaces) was given by Coifman and
the first author \cite{coif1}. Ways of computing the expansion for non-analytic signals are due to Coifman and the authors \cite{coif2}.  An extension to Hardy spaces $H^p$ is due to Coifman and Peyri\`{e}re \cite{jac}. The algorithm seems to have exceptional
convergence properties when applied to real signals, but a full theoretical justification is still open.
\begin{center}
\begin{figure}[h!]
\includegraphics[width = 0.9\textwidth]{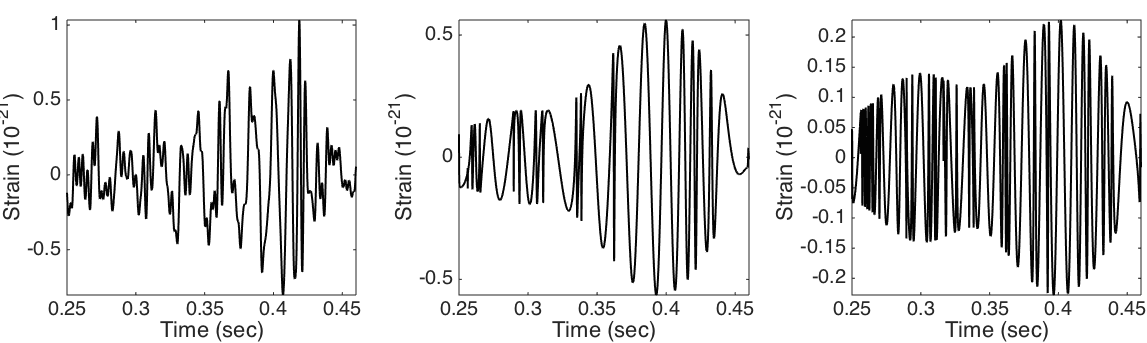}
\caption{The signal of a gravity wave (left), the first Blaschke product of the respective signal (middle) and the second Blaschke
product (right), from \cite{coif2}.}
\end{figure}
\end{center}

\textbf{Polynomials.}
If the function $f$ is a polynomial of degree $n$, then the expansion is exact after $n$ steps (this was already observed by Nahon \cite{nahon}). For polynomials, the explicit form of the Blaschke products
allows for the algorithm to be described in a simpler way: given a polynomial $f_n$
\begin{enumerate}
\item define the polynomial $g_{n+1}$ to be the polynomial having the same roots as $f_n$ outside the unit disk and, additionally, the roots $1/\overline{\alpha}$ for all roots $\alpha \neq 0$ of $f$ inside the unit disk, i.e.
$$ g_{n+1} = f_n(z)   \left(\prod_{|\alpha| \leq 1, f(\alpha) = 0}{ \frac{z- \alpha}{1 - \overline{\alpha} z}} \right)^{-1}$$
\item define $f_{n+1}(z) = g_{n+1}(z) - g_{n+1}(0)$ and, if $f_{n+1} \neq 0$, go to (1)
\end{enumerate} 

The main question is with which speed $f_n$ converges to 0 on the boundary of the unit disk. The paper \cite{coif1} shows that convergence speed in the Dirichlet space
can be explicitly connected to how many roots inside the unit disk one would expect $f_n$ to have. Using Theorem 1 of this paper, we can answer the question from \cite{coif2}
and conclude that for typical polynomials (and $n$ large), one cannot expect more than $o(n)$ roots inside the unit disk. 
\begin{corollary}[Invariance of certain random polynomials under Blaschke factorization.] Let $p_n$ be a random polynomial with $n$ roots that are independently and identically distributed following a probability measure that can be written as $\mu = \phi(\sqrt{x^2+y^2})dx dy$ for some $\phi \in C^{\infty}_c((1,\infty))$. For every such polynomial $p_n$, we may determine the Blaschke (or inner-outer) Factorization
$$ p_n(z) = p_n(0) + B \cdot G.$$
Then, $G$ is a random polynomial whose roots are also distributed according to $\mu$ as $n \rightarrow \infty$.
\end{corollary}

More precisely, let $p_n$ be a random
polynomial created in the way described at above for some radial probability measure $\mu$ that is compactly supported outside a neighborhood of the unit disk. Then Theorem 1 implies that the
roots of $p_n(z) - p_n(0)$ are again distributed according to the measure $\mu$ as $n \rightarrow \infty$. The proof of Theorem 1 also implies that with high probability all solutions of $p_n(z) = p_n(0)$
except the trivial root in the origin are outside the unit disk since they are exponentially close to the $n$ roots with high likelihood.
We observe that in this case, when $n$ is sufficiently large, the Blascke unwinding series reduces to a simple power series expansion. A similar phenomenon was already observed to occur for functions whose power series expansion has 
exponentially decaying coefficients in \cite[Proposition 3.2]{coif2}. {It seems likely that polynomials with roots outside the unit disk exhibit exponentially decaying coefficients at least in the generic case -- simple power series expansion then naturally leads to exponentially convergence in the unit disk.}

\section{Proofs}
We start by first proving a concentration of measure argument that plays a role in all three proofs. We then prove Theorem 2. Theorem 1 will follow from a small modification of the same argument. Theorem 3 follows from a different line of reasoning.\\

The whole argument is based on {establishing} the fact that if $p_n(z) = 0$ for some $z \in A$, then with high probability there is a solution of $p_n(z) = p_n(0)$ that is exponentially close (in the degree $n$) to $z$. We start from giving a heuristic argument, which will be made rigorous in the following proof. First, for any fixed $z \in \mathbb{C}$ the distance
to the nearest root is at scale $\sim n^{-1/2}$ in the sense that it is not going to be closer, but it might be further away. Note that if $z$ is not in the support of $\mu$, then this statement is trivial.
On the other hand, by a direct expansion
we expect
$$ \log{|p_n(z)|}\sim n  \int_{\mathbb{C}}{ \log{|z-x|} d\mu(x)}.$$
{Since the entire theorem is invariant under scaling all the roots by (the same) scalar $\lambda \in \mathbb{R}$, we may assume without loss of generality that this integral is positive.}
 Based on the above two facts, due to the root separation at scale $\sim n^{-1/2}$, since $\log{|p_n(z)|} =  \sum_{k=1}^{n}{ \log{|z-z_k|}}$, in order for a {\em single} root $z^*$ to substantially contribute to $\log{|p_n(z)|}$, we would require that $\log{|z-z^*|} \sim-n $ which requires that $z^*$ is exponentially close to $z$.

\subsection{Two Concentration Lemmata}
In the first lemma, we will actually prove a somewhat stronger statement; as it turns out, the likelihood of the quantity exceeding
the logarithmic integral is uniformly small.

\begin{lemma}\label{Lemma 5} Let $\mu$ and $p_n$ be as above. Let $c_1 > 0$, then, for some $c_2, c_3 > 0$,
$$ \mathbb{P}\left( \sup_{z \in \mathbb{C}} \left( \frac{\log{|p_n(z)|}}{n}  -  \int_{\mathbb{C}}{ \log{|z-x|} d\mu(x)} \right) \geq c_1 \right)  \leq c_2 e^{-c_3 n}$$
\end{lemma}
\begin{proof} 

Our assumptions on $\mu$ imply that for any fixed $z \in \mathbb{C}$, $\int_{\mathbb{C}}{ \log{|z-x|} d\mu(x)}$ is well defined, continuous in $z$, and finite everywhere (and, as can easily be seen, this would also hold for non-compactly supported
measures that decay with a certain rate). 
We split the support of $\mu$ into $\varepsilon \times \varepsilon$ squares. Since the support is compact, the total number
of squares is finite. Each of these squares $Q$ receives an expected number of $\mu(Q)n \pm \mathcal{O}(\sqrt{n})$ roots and the likelihood of a square receiving more than
$(1+ \delta)\mu(Q) n$ roots or less than $(1-\delta)\mu(Q)n$ roots is exponentially decaying in $n$ for every $\delta >0$ (with a constant in the exponential decay depending on $\mu(Q)$ and $\delta$). An application of the union bound then implies that all the squares are exponentially likely to have at most $(1+\delta)n$ roots. 
We can now bound $\log{|p_n(z)|}$ directly: we assume that there are no roots in the $\varepsilon \times \varepsilon$ square containing $z$ as well as no roots in the 8 adjacent squares. This introduces an arbitrarily small error (depending on $\varepsilon$ because the logarithmic integral is integrable); for the remaining squares, we assume that the roots are located in each box as far away from $z$ as possible. 
Outside of a neighborhood of $z$, we can use the continuity of the logarithm to deduce that the errors are small (depending on $\delta, \varepsilon$), in a neighborhood of $z$ we can use that the integral becomes as small as we wish because it is an integrable singularity. Altogether, we obtain (with uniform bounds depending only on the distribution of number of roots inside the finitely many squares)
$$ \mathbb{P}\left( \sup_{z \in \mathbb{C}} \frac{\log{|p_n(z)|}}{n}  -  \int_{\mathbb{C}}{ \log{|z-x|} d\mu(x)} \geq c_1 \right) \leq c_2 e^{-c_3 n},$$
where $c_2, c_3$ depend on $\varepsilon$ and $\mu$. 
\end{proof}

No such universal statement can be true for the corresponding lower bound because
$p_n$ will have roots where the logarithm tends to $-\infty$. However, this is the only obstruction to uniform approximation and we will now show that outside the 
roots, even already exponentially small distances away, the approximation is again uniform with high probability.

\begin{lemma}\label{Lemma 6} Let $z \in \mathbb{C}$ and $c_1, c_2>0$ be fixed. Then, for any $c_3 > 0$ and all $n$ sufficiently large (depending on all previous parameters)
$$ \mathbb{P}\left( \frac{\log{|p_n(z)|}}{n}  -  \int_{\mathbb{C}}{ \log{|z-x|} d\mu(x)} \leq -c_1 ~\Big|~ \emph{no roots in}~B\left(z,\frac{c_2}{\sqrt{n}}\right) \right) \leq \frac{c_3}{n}\,,$$
where $B(z,r)$ denotes the ball with center $z\in \mathbb{C}$ and radius $r>0$.
\end{lemma}
\begin{proof} From the proof of Lemma 5 we infer that the only roots that are relevant are those close to $z$. This is because for roots that are far away, their number exhibits exponential concentration in the associated boxes and the logarithm is continuous.
Fix $0 < \delta < 1$ arbitrarily small. We will now show that the terms coming from roots with distance to $z$ less than $\delta$ can be controlled. Take annuli 
$$A_k=B(z, c_2 (k+1) n^{-1/2}) \setminus B(z, c_2 k n^{-1/2}).$$ 
The contribution to the logarithm of the polynomial depends on the likelihood of roots landing in these annuli and is given by
$$ X_k := \sum_{\ell=1}^{n}{1_{x_{\ell} \in A_k} \log{|x_{\ell} - z|}}.$$
Instead of bounding $X_k$ from below, we will bound $-X_k$ from above (purely for simplicity of exposition). 
We observe that the annulus $A_k$ has a measure of
$$ |A_k| \leq \frac{\pi c_2^2 (2k+1)}{n}$$
and thus the expectation of points landing in $A_k$ can be bounded in terms of the biggest density of $\mu$ via the Radon-Nikodym derivative $d\mu/dx$
\begin{align*}
\mathbb{E} (-X_k) &\leq  \log{\left(\frac{\sqrt{n}}{c_2 k}\right)}  \mathbb{E} \sum_{\ell=1}^{n}{1_{x_{\ell} \in A_k}} \leq  \log{\left(\frac{\sqrt{n}}{c_2 k}\right)}  |A_k| n \left\| \frac{d \mu}{dx}\right\|_{L^{\infty}} \lesssim_{\mu}  \log{\left(\frac{\sqrt{n}}{c_2 k}\right)} c_2^2 k\,,
\end{align*}
where $\log{\left(\sqrt{n}/(c_2 k)\right)}$ in the first bound comes from $|z-x_k|>c_2 k n^{-1/2}$ since $x_k\in A_k$.
We also compute the variance.  Note that we can view $1_{x_\ell\in A_k}$ in $X_k$ as a binomial distribution with the parameters $n$ and $p=\mu(A_k) \leq |A_k|\|d\mu/dx\|_{L^{\infty}}$.  By assumption, $\{1_{x_\ell\in A_k}\}_{\ell=1}^n$ are independently and identically sampled. Thus, using that $\mathbb{V}(\lambda X) = \lambda^2 \mathbb{V}(X)$ as well as that for a binomial distribution $B(n,p)$ we have $\mathbb{V} B(n,p) = np(1-p) \leq np$, we see that
$$ \mathbb{V} (-X_k) \leq   \left(\log{\left(\frac{\sqrt{n}}{c_2 k}\right)}\right)^2 \mathbb{V} \left( \sum_{\ell=1}^{n}{1_{x_{\ell} \in A_k}} \right) \lesssim_{\mu} c_2^2 k \left(\log{\frac{\sqrt{n}}{c_2 k}}\right)^2 .$$
We will now control the sum over $X_k$,  where $k \leq \delta \sqrt{n}/c_2$ (this corresponds to a disk of radius $\delta$ around $z$) and obtain
$$ \sum_{k \leq \delta \sqrt{n}/c_2}{ \mathbb{E}( -X_k)} \lesssim  \sum_{k \leq \delta \sqrt{n}/c_2}{  c_2^2 k \log{ \left(  \frac{\sqrt{n}}{c_2 k} \right)} } \lesssim \int_{0}^{\delta \sqrt{n}}{ x \log{\frac{\sqrt{n}}{x}} dx} \lesssim \delta^2 \log{\left(\frac{1}{\delta}\right)} n\,.$$
We now want to obtain a similar bound on the variance and proceed as follows
\begin{align*}
\mathbb{V}\left(  \sum_{k \leq \delta \sqrt{n}/c_2}{ -X_k}\right) &= \mathbb{E} \left( \sum_{k \leq \delta \sqrt{n}/c_2}{ X_k} - \mathbb{E} \sum_{k \leq \delta \sqrt{n}/c_2}{ X_k}\right)^2 \\
&= \sum_{k, \ell \leq \delta \sqrt{n}/c_2} \mathbb{E}(X_k - \mathbb{E} X_k) \mathbb{E}(X_{\ell} - \mathbb{E}(X_{\ell}))
\end{align*}
The diagonal terms are computed above and correspond to $\mathbb{V}(-X_k)$. The off-diagonal terms $k \neq \ell$ are easy to deal with: if $X_k$ and $X_{\ell}$ were uncorrelated, then these terms would simply be 0. They are not perfectly uncorrelated but exhibit a (very) slight negative correlation:
pulling out the contribution coming from the logarithm, we reduce the problem to studying the following simpler problem: we are given $n$ points and distribute them in several boxes, what can be said about the cross-correlation? If one box receives unexpectedly many points, then there are fewer points left over to distribute over the other boxes and we expect them to get less than their expectation; if one box receives unexpectedly few points, then the argument reverses. Altogether, we see that the cross-correlation is negative and thus
$$\mathbb{V}\left(  \sum_{k \leq \delta \sqrt{n}/c_2}{ -X_k}\right)  \lesssim \sum_{k \leq \delta \sqrt{n}/c_2}{ c_2^2 k \left( \log{\frac{\sqrt{n}}{c_2k} } \right)^2}  \lesssim \delta^2 \left(\log{\frac{1}{\delta}}\right)^2 n.$$
The bound on the expectation shows that, asymptotically for $\varepsilon$ small, we can find $\delta$ sufficiently small so that the contribution to the term of interest is arbitrarily small (this mirrors the
fact that the logarithmic integral is integrable and so are all logarithmic integrals with integer powers on the logarithm). We now use Chebyshev's inequality to derive that the likelihood of the contributions coming from roots that have their distance from $z$ bounded by $\delta$ to exceed a constant $c_1$
$$ \mbox{ can be bounded from above by} \lesssim  \frac{\delta^2 \left(\log{\frac{1}{\delta}}\right)^2 n}{c_1^2 n^2}.$$
The roots further away than distance $\delta$ can be dealt with by appealing to continuity of the logarithm outside a neighborhood of $z$ together with the exponential localization of the number of roots in boxes akin to the proof of Lemma 5. Since $\delta$ can be chosen arbitrarily small, we obtain the result.
\end{proof}

\subsection{Proof of Theorem 2}
We start by performing a standard geometric estimate to argue that few roots have another root nearby, meaning at distance $c_2 n^{-1/2}$, whenever $c_2$ is small. Clearly, the largest concentration occurs if the measure has constant density $\|d \mu/dx\|_{L^{\infty}}$ in which case 
$$ \mathbb{P}\left(\mbox{no roots in}~B(z, c_2 n^{-1/2})\right) = \left(1 - \left\| \frac{d\mu}{dx}\right\|^2 \frac{c_2^2}{n}\right)^{n}.$$
For $n$ large, we can approximate this with the exponential function and conclude that the likelihood scales like $\exp(- c_2^2 \|d\mu/dx\|^2_{L^{\infty}} )$. In particular, for $c_2$ sufficiently small, an arbitrarily small proportion of roots has another root nearby (meaning at distance less than $c_2 n^{-1/2}$). For these roots, the likelihood of deviating from the logarithmic integral is very small and only
$\mathcal{O}(1)$ of isolated roots will do so. For the rest, we can use Rouch\'{e}'s theorem to conclude that each but $\mathcal{O}(1)$ of the isolated roots of $p_n(z)$ is exponentially
close to a root of $p_n(z) - p_n(0)$ which leads to a reproduction of measure.

\begin{proof}
The proof is based on understanding the expected size of $\mathbb{E} n^{-1}\log{|p_n|}$. 
For any fixed $z \in \mathbb{C}$,
$$ \mathbb{E} ~ \frac{\log{|p_n(z)|}}{n} =\mathbb{E} \frac{1}{n}  \sum_{k=1}^{n}{ \log{|z-z_k|}} = \int_{\mathbb{C}}{ \log{|z-x|} d\mu(x)}.$$
%Our assumptions on $\mu$ imply that this integral is well defined, continuous in $z$, and finite everywhere (and, as can easily be seen, this would also hold for non-compactly supported
%measures that decay with a certain rate). 
Let us now assume that $z \in A$.  
Since $\mu$ has an
absolutely continuous probability measure associated with the Lebesgue measure $dx$ with a continuous distribution function, if we subdivide the support of the measure into finitely many boxes of equal size, we know that for $n$ sufficiently large, each box contains a number of roots proportional to the measure assigned to that box by $\mu$. By Lemmas \ref{Lemma 5} and Lemma \ref{Lemma 6}, we have a large deviation principle: for any $c>0$ and $z \in \mathbb{C}$
\begin{equation}\label{LargeDeviation1}
\mathbb{P}\left(\left|\frac{\log{|p_n(z)|}}{n}  -  \int_{\mathbb{C}}{ \log{|z-x|} d\mu(x)}\right| \geq c \right) ~ \mbox{is decaying in $n$}
\end{equation}
when there is no root in a sufficiently small neighborhood of $z$.
In particular, the likelihood of $\log{|p_n(z)|}$ being actually bigger than the logarithmic integral are exponentially small, the likelihood of it being smaller is polynomially small assuming one is
distance $\sim n^{-1/2}$ away from the roots. \\

For any given root $y\in A$, we can remove it and write $q_n(z) (z-y) = p_n(z)$. 
We now pick $c_2$ in Lemma \ref{Lemma 6} arbitrarily small and $n$ sufficiently large.  By a union bound argument with $q_n(z)$, this guarantees that with high probability all but a small proportion of roots are actually $c_2 n^{-1/2}$ away from a root $y$. %
Lemma \ref{Lemma 5} and Lemma \ref{Lemma 6} and the definition of $A$ imply that for isolated roots $|q_n(z)| \gtrsim (1+\eta)^n |p_n(0)| $ for some $\eta>0$ depending
on the location of $y \in A$ and $n$ for all $z$ sufficiently close to $y$. 
Thus, $|p_n(z)|=|z-y||q_n(z)|\gtrsim |z-y|(1+\delta)^n |p_n(0)|$. In an exponentially small (depending on $\delta$) disk of radius $\delta'>0$ around the root $y$, $|z-y|(1+\delta)^n$ is sufficiently small so that $|p_n(z)|>|p_n(0)|$ for all $z$ on the boundary of $B(y,\delta')$. 
This bound holds for all but $\mathcal{O}(1)$ of isolated roots. Rouch\'{e}'s theorem applied in an exponentially small (depending on $\delta$) disk around the root $y$ then implies that $p_n(z) -p_n(0)$ has a root 
in that small disk and this implies the result.\\

The second part of the statement is much simpler: if $z \in \mathbb{C} \setminus A \cup B$, then this means that
$$ \int_{\mathbb{C}}{ \log{|x-z|} d\mu(x)}  < \int_{\mathbb{C}}{ \log{|x|} d\mu(x)} = \mathbb{E} |p_n(0)|^{1/n}.$$
Lemma 5 shows that the likelihood of the left-hand side exceeding its expectation is exponentially smart. Lemma 6 implies that the likelihood of the right-hand side being a lot smaller hinges on a root being nearby. However, as $n$ becomes large, that root would have to be exponentially close 0 to compensate for difference in expectation and that yields the desired statement. 
\end{proof}

\subsection{Proof of Theorem 1.}
\begin{proof} We use Theorem 2 and compute the sets $A$ and $B$.
We start by showing that for radial measures $\mu$, the function
$$ \int_{\mathbb{C}}{ \log{|z-x|} ~d\mu(x)} \qquad \mbox{has a global minimum in the origin.}$$
This can be seen rather easily from the elementary observation that
$$ \frac{1}{2\pi} \int_{0}^{2\pi}{ \log{|z - r e^{it}|} dt} = \begin{cases} \log{|z|} \qquad \mbox{if}~|z| > r \\
\log{|r|} \qquad \mbox{if}~|z| < r. \end{cases}$$
Using $\phi(r)$ to denote the Radon-Nikodym derivative of $\mu$ with respect to the Lebesgue measure, {\color{red}we} can write
\begin{align*}
 \int_{\mathbb{C}}{ \log{|z-x|} d\mu(x)}& = \int_{0}^{\infty}{ \phi(r) r \int_{0}^{2\pi}{ \log{|z - r e^{it}|} ~dt} dr} \\
&=  2\pi  \int_{0}^{|z|}{ \phi(r) r \log{|z|}~ dr} + 2\pi  \int_{|z|}^{\infty}{ \phi(r) r \log{r} ~dr} \\
&=  \int_{\mathbb{C}}{\log{|x|} d\mu(x)}+  2\pi  \int_{0}^{|z|}{ \phi(r) r \log{\frac{|z|}{r}} ~dr}.
\end{align*}
The second integral is always nonnegative, this shows that there is a global minimum in $z = 0$. It also allows
us to determine
$$ B = \begin{cases} B(0,R) \qquad &\mbox{if}~\phi \equiv 0~\mbox{on}~(0,R) \\
\left\{ 0 \right\} \qquad &\mbox{otherwise} \end{cases} \qquad \mbox{and} \qquad A = \mathbb{C} \setminus B.$$
This implies that $\mu(B) = 0$ and
Theorem 2 then implies that the density accurately reproduces $\mu$ on $A$. This implies the result.
\end{proof}

\subsection{Proof of Theorem 3.}

\begin{proof}
Let $p_n(z)$ be a random polynomial. 
The set $C$ will be a natural limit set for the (random) set $C_n$ associated to a random polynomial $p_n$ and defined by
$$ C_n = \left\{z \in \mathbb{C}: \sum_{k=1}^{n}{ \log{|z-z_k|}} = \sum_{k=1}^{n}{ \log{|z_k|}} \right\}.$$
$C_n$ is the level set of a superposition of random functions and does a priori look quite complicated. However, since we will only be studying it away from the support of $\mu$ in a neighborhood of $C$ and recall the deviation principle from the proof of Theorem 2, we see that these
objects are rather rigid. On parts of $C_n$ that are uniformly bounded away from the support of $\mu$, we see that
$$ \lim_{n \rightarrow \infty}{  n^{-1}\sum_{k=1}^{n}{\nabla \log{|z-z_k|}}} =   \nabla \int_{\mathbb{C}}{ \log{|x-z|} d\mu(z)}  \qquad \mbox{in probability.}$$
Moreover, by the same argument this extends to higher derivatives on $C$ since all higher derivatives are uniformly bounded (because $C$ is supported away from the support of $\mu$). This shows that for $n$ sufficiently
large, with high probability $C_n$ is a curve (a segment of which converges uniformly (together with its derivatives) to $C$ as $n \rightarrow \infty$). Let us assume that $\gamma_n$ is an arclength parametrization of a segment of $C_n$
on which the assumptions of Theorem 3 apply. $\gamma_n$ then parametrizes a curve on which $|p_n(z)| = |p_n(0)|$. It remains to see whether the arguments of the complex numbers can be matched to produce a solution of the equation. We note that
$$ \frac{\partial}{\partial t} \arg \prod_{k=1}^{n}{(\gamma_n(t)-z_k)} = n \frac{\partial}{\partial t} \frac{1}{n}\sum_{k=1}^{n}{ \arg (\gamma_n(t)-z_k)}.$$
For $n$ sufficiently large, this quantity converges to 
$$ \frac{\partial}{\partial t} \frac{1}{n}\sum_{k=1}^{n}{ \arg (\gamma_n(t)-z_k)} \rightarrow  \frac{\partial}{\partial t} \int_{\mathbb{C}}{\arg(\gamma(t) - z) d\mu(z)} \qquad \mbox{in probability,}$$
where $\gamma(t)$ is some curve satisfying $\gamma'(t) = \lim_{n \rightarrow \infty} \gamma_n'(t)$ (this, of course, leads exactly to an arclength parametrization of $C$).
This shows that the argument is asymptotically moving linearly in $n$. Therefore, when $n$ is sufficiently large, with high probability, the argument of $p_n(\gamma(t))$ hits the argument $p_n(0)$ at a rate given by a continuous function. As a result, we have a regular distribution of solutions of the equation along the level set: the argument needs to complete a total
revolution of $2\pi$ which accounts for the arising pre-factor. Since the linear rate is $\left|\frac{\partial}{\partial t} \int_{\mathbb{C}}{\arg(\gamma(t) - z) d\mu(z)}\right|$, the associated measure on $C$ is thus described in (\ref{Theorem:eq1}). 
\end{proof}

It is not difficult to see that the argument can be extended to the setting where $C$ and the measure of $\mu$ are not disjoint (but $\mu$ is still assumed to be absolutely continuous
with respect to the Lebesgue measure): the random curve $C_n$ is only minorly impacted by roots nearby (which would need to be exponentially close to have an impact which
becomes increasingly unlikely), we leave the details to the interested reader.

\subsection{An explicit example.}
This section is devoted to an explicit computation for what to expect in the example
$$ \mu = \frac{1}{4\pi} \left( \mathcal{H}^1\big|_{|z| = 1} \cup  \mathcal{H}^1\big|_{|z-2| = 1} \right)$$
(see Fig. 2). Summarizing the proof, we can fix a point $z \in \mathbb{C}$ and compute
$$ \mathbb{E} n^{-1} \log{|p_n(z)|} = \frac{1}{n} \sum_{k=1}^{n}{ \log{|z-z_k|}} \rightarrow \int_{\mathbb{C}}{ \log{|z-x|} d\mu(x)}$$
because the likelihood of having singularities nearby is small. Moreover, we have
$$ \frac{1}{2\pi} \int_{0}^{2\pi}{ \log{|z -  e^{it}|} dt} = \begin{cases} \log{|z|} \qquad \mbox{if}~|z| > 1 \\
0\qquad\qquad \mbox{if}~|z| < 1. \end{cases}$$
Thus,
$$  
\int_{\mathbb{C}}{ \log{|z-x|} d\mu(x)} = \begin{cases}
\frac{1}{2} \log{|z-2|} \qquad &\mbox{if}~|z| < 1 \\
\frac{1}{2} \log{|z|}  \qquad &\mbox{if}~|z-2| < 1 \\
\frac{1}{2} \log{|z|} + \frac{1}{2} \log{|z-2|}  \qquad &\mbox{otherwise.}
\end{cases}
$$
This also shows that we expect exponential growth in the origin
$$ \mathbb{E} n^{-1} \log{|p_n(0)|} = \frac{\log{2}}{2}.$$
It remains to find all points in the complex plane for which the logarithmic integral equals that quantity and those are displayed in Figure 2.

\end{document}